\documentclass[a4paper,11pt]{amsart}
\newif\ifcolor\colorfalse
\pdfoutput=1
\usepackage{amsmath,amssymb}
\usepackage{amsthm}
\usepackage{amsfonts}
\usepackage[utf8]{inputenc}
\usepackage{enumerate}
\usepackage{hyperref}
\usepackage{paralist}
\usepackage{url}
\usepackage{color}
\usepackage{framed}
\usepackage{mathrsfs}
\usepackage{graphicx}
\usepackage{bm}
\usepackage{lipsum}                     
\usepackage{xargs}                      
\usepackage{xcolor}  
\usepackage{float}
\usepackage[all,pdf]{xy}
\usepackage{mathtools}
\usepackage[ruled,vlined]{algorithm2e}
\usepackage{subcaption}

\newtheorem{theorem}{Theorem}[section]
\newtheorem{lemma}[theorem]{Lemma}
\newtheorem{corollary}[theorem]{Corollary}

\newtheorem{proposition}[theorem]{Proposition}
\newtheorem{definition}[theorem]{Definition}

\theoremstyle{definition} 
\newtheorem{example}[theorem]{Example}
\newtheorem{remark}[theorem]{Remark}

\newtheorem*{openproblem}{Qusetion}

\newcommand{\dual}{\mathrm{o}}

\newcommand{\VV}{\mathbf{V}}

\newcommand{\ideal}[1]{\langle #1 \rangle}
\newcommand{\QQ}{\mathbb{Q}}
\newcommand{\RR}{\mathbb{R}}

\newcommand{\PP}{\mathbb{P}}
\newcommand{\CC}{\mathbb{C}}

\newcommand{\sing}{\text{\textup{Sing}}}

\newcommand{\cx}{C(X)} 
\newcommand{\iy}{\mathcal{I}_{Y}} 
\newcommand{\icx}{\mathcal{I}_{C(X)}} 
\newcommand{\icy}{\mathcal{I}_{C(Y)}} 
\newcommand{\icxy}{\mathcal{I}_{C(X)\cap C(Y)}} 
\newcommand{\xzero}{\mathrm{Reg}(X)}
\newcommand{\yzero}{U}
\newcommand{\icxyz}{\mathcal{I}_{C(X)\cap C(\yzero)}} 
\newcommand{\ikxy}{\mathcal{I}_{\kappa^{-1}_X(Y)}} 
\newcommand{\ikxyz}{\mathcal{I}_{\kappa^{-1}_X(\yzero)}} 
 \newcommand{\setcap}{\cap} %

\newcommand{\Sing}{\mathrm{Sing}}
\newcommand{\Reg}{\mathrm{Reg}}
\newcommand{\codim}{\mathrm{codim}}
\newcommand{\Ex}{\mathrm{Ex}}
\newcommand{\Exr}{\mathrm{Exr}}
\newcommand{\co}{\mathrm{co}}
\renewcommand{\int}{\mathrm{int}}
\newcommand{\cls}{\mathrm{cls}}

\newcommand{\NN}{\mathbb{N}} 
\renewcommand{\AA}{\mathbb{A}}
\newcommand{\GG}{\mathbb{G}}

\newcommand{\<}{\langle}
\renewcommand{\>}{\rangle}

\usepackage{graphicx}
\title[Whitney Stratification of  Algebraic Boundaries]{Whitney Stratification of  Algebraic Boundaries  of Convex Semi-algebraic Sets}
\author{Zihao Dai, Zijia Li, Zhi-Hong Yang, Lihong Zhi}
\keywords{Dual variety, Conormal space, Normal cone, Convex semi-algebraic set, Extreme point, Whitney Stratification}
\subjclass[2010]{52A99, 14N05, 14P10, 51N35, 14Q15}

\begin{document}

\begin{abstract}
Algebraic boundaries of convex semi-algebraic sets are closely related to polynomial optimization problems.  Building upon  Rainer Sinn's work, we refine the stratification of iterated singular loci to a Whitney (a)
stratification, which gives a list of candidates of varieties whose dual is an irreducible component of the algebraic boundary of the dual convex body. We also present an algorithm based on Teissier's criterion to compute Whitney~(a) stratifications, {\ifcolor\color{magenta}\fi which employs conormal spaces and prime decomposition. }

\end{abstract}

\maketitle

\setcounter{tocdepth}{1}

\section{Introduction}\label{sec:introduction}

Let $K$ be a convex semi-algebraic compact set with $0$ in its interior. Let $\partial K$ be the Euclidean boundary of $K$. The algebraic boundary of $K$, denoted $\partial_aK$, is the Zariski closure of $\partial K$. The convex hull of a real algebraic variety has important applications in optimization \cite{blekherman2012semidefinite,nie2010algebraic,gouveia2012convex,gouveia2010theta,wang2019global}.

The algebraic boundary of the convex hull of a compact real algebraic variety 
\[V=\{x\in \RR^n\ |\ f_j(x)=0,\ f_j\in \RR[x],\ j=1,\ldots,r\}\]
has been studied by Ranestad, Rostalski, and Sturmfels \cite{ranestad2011convex,ranestad2012convex,rostalski2012chapter}.  
 If $V$ is an irreducible and smooth compact variety, by \cite[Theorem 1.1]{ranestad2011convex}, the algebraic boundary of its convex hull can be computed by biduality.   Guo et al. extended this result to non-compact or non-smooth real algebraic variety  \cite{MR3388301}. 
   In \cite{Sinn2015},
  Sinn studied the algebraic boundary of a convex semi-algebraic set.
 In \cite[Corollary 3.4]{Sinn2015}, he proved that the dual of an irreducible component of $\Ex_a(K)$ (the Zariski closure of extreme points of $K$, see Definition~\ref{def:exa}) is an irreducible component of $\partial_aK^\dual$ ($K^\dual$ is the dual convex body of $K$, see Definition \ref{def:convex}). However, the converse may not hold. Namely, the dual of an irreducible component of $\partial_aK^\dual$ may not be an irreducible component in $\Ex_a(K)$.  
 In  \cite[Theorem 3.16]{Sinn2015}, Sinn showed that the iterated singular locus of $\partial_aK$ gives a list of candidates of irreducible subvarieties $\overline Z$ (the projective closure of $Z$, see Definition~\ref{projclos})
 of $\overline{\Ex_a(K)}$, whose dual variety $\overline{Z}^*$ is an irreducible component of $\overline{\partial_aK^\dual}$. 
%
There is a condition contained in   \cite[Theorem 3.16]{Sinn2015}  that requires every point on the boundary of the convex set to be regular on each irreducible component of 
the
algebraic boundary containing it, and Sinn gave a counterexample \cite[Example 3.20]{Sinn2015} when the condition is removed. 


\begin{example}
\label{ex:tear}{\upshape \cite[Example 3.20]{Sinn2015}}
Let
  \[f=(z^2+y^2-(x+1)(x-1)^2)(y-5(x-1))(y+5(x-1)) \in \QQ[x,y,z].\] 
   The real variety of $I=\langle f \rangle$ 
   is shown in Figure~\ref{fig:teardrop1}.
{\ifcolor\color{teal}\fi
The \emph{teardrop} defined by 
$$
K=\{(x,y,z) \in \RR^3 ~|~ z^2+y^2-(x+1)(x-1)^2\le 0,\ x\le 1\}
$$
is a convex semi-algebraic set. The convexity of $K$ can be checked by its Hessian matrix.
}
\begin{figure}[tbhp] 
  \centering 
  \begin{subfigure}[t]{0.51\textwidth}
  \centering
      \includegraphics[width=0.8\textwidth]{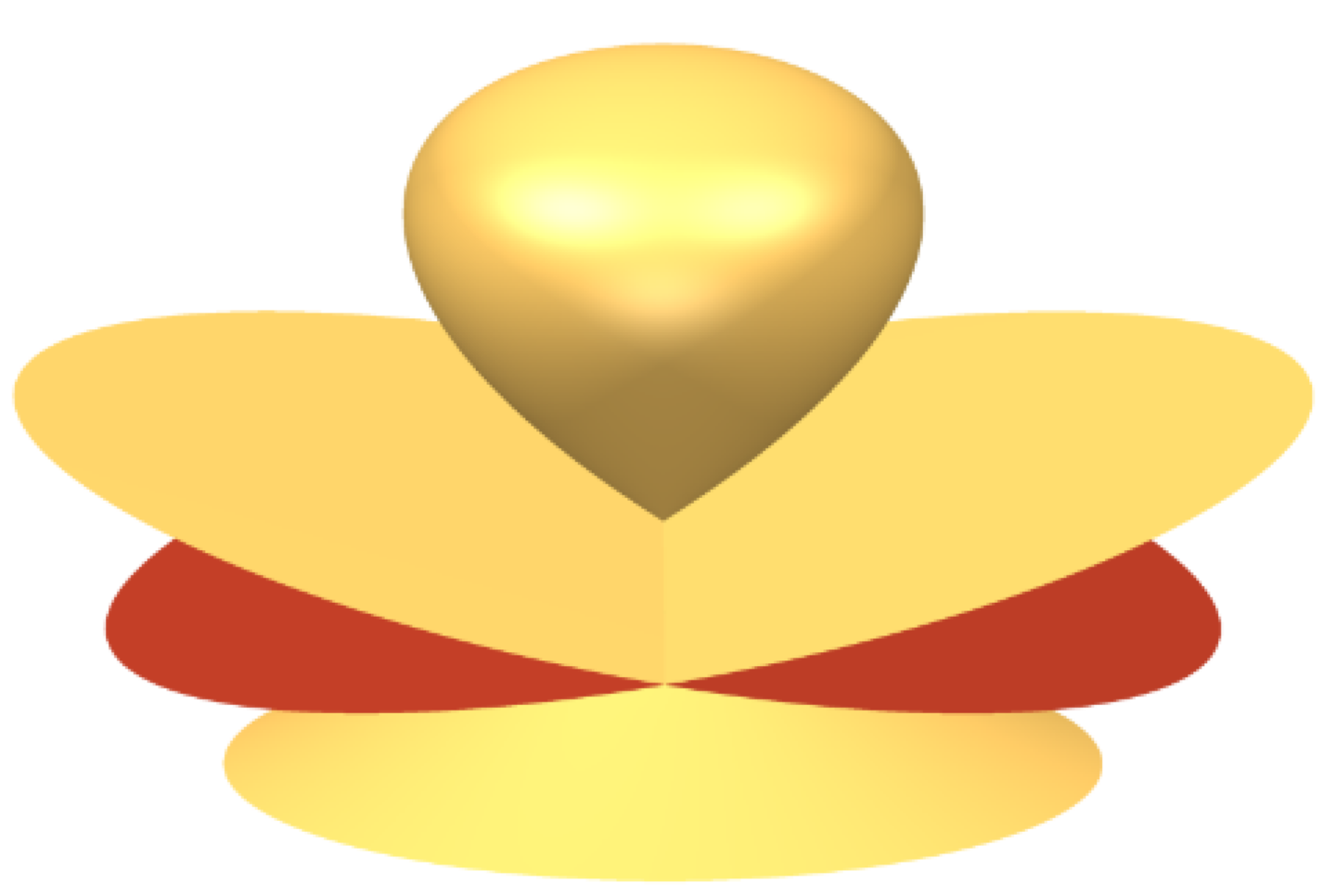} 
  \caption{
  {\ifcolor\color{teal}\fi
  The real variety of $f$ which contains the algebraic boundary of the teardrop $K$.
  }}
  \end{subfigure}$\quad$
  \begin{subfigure}[t]{0.45\textwidth}
  \centering
  \includegraphics[width=0.8\textwidth]{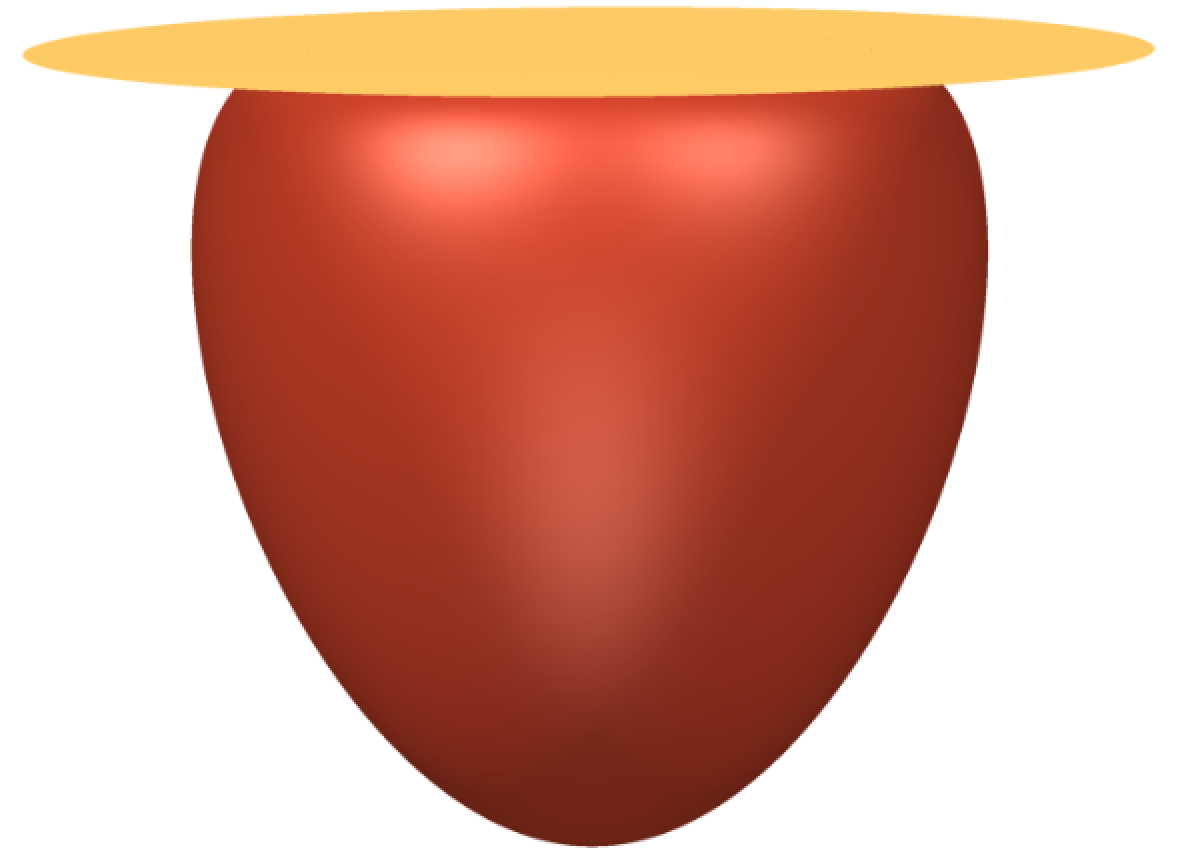} 
  \caption{{\ifcolor\color{teal}\fi
  The dual convex body of the teardrop $K$.
  }}
  \end{subfigure}
  \caption{Example~\ref{ex:tear}.}
  \label{fig:teardrop1}  
\end{figure} 



Let $g=f+\frac{1}{10}(x-1)yz^2$ 
{\ifcolor\color{magenta}\fi and }
$\VV(g)$ be the real variety defined by the polynomial $g$.
{\ifcolor\color{magenta} Then  $\VV(g)$ is the algebraic boundary of a perturbed teardrop $K'$, that is, $\VV(g) = \partial_a K'$. \fi}
The singular locus of $\partial_aK'$ is the union of $\VV(I_1),\VV(I_2),\VV(I_3)$ where
\begin{align*}
  I_1=&\langle {y,x-1}\rangle,
\\ 
 I_2=&\langle {z,-115+y,x-24}\rangle,
\\ 
 I_3=&\langle {z,115+y,x-24}\rangle.
\end{align*}
The 2nd singular locus of $\VV(g)$ is an empty set. 
{\ifcolor\color{teal}\fi
The extreme point $p=(1,0,0)$ of $K'$ lies on the line $\VV(I_1)$, but $p$ is not contained in any irreducible component of the iterated singular loci.
}
On the other hand, as pointed out by Sinn, the normal cone of the point $(1,0,0)$ relative to $K'$ is of dimension $3$.
The dual of $Z=\{(1,0,0)\}$ is the hyperplane $\VV(x+1)$, which is an irreducible component of $\partial_a(K')^\dual$ {\upshape\cite[Corollary 3.9]{Sinn2015}}. 


\end{example}

Sinn's \cite[Example 3.20]{Sinn2015} shows that the 
point $(1,0,0)$  can be discovered by 
checking Whitney's condition~(a) for the pair
\[(\Reg(\VV(g)),\VV(I_1)).\]
Hence, it would be interesting to ask the following question: 

\begin{openproblem}{\upshape \cite[Remark~3.17$(b)$]{Sinn2015}}
    Is Whitney's condition (a) sufficient for discovering all subvarieties $\overline{Z}\subseteq\overline{Ex_a(K)}$ whose dual is an irreducible component in $\overline{\partial_aK^\dual}$?
\end{openproblem}


 
 We give a positive answer to this question by proving the following theorem:

\begin{theorem}[Theorem \ref{thm:strata}]
	Let $K\subseteq\RR^n$ be a semi-algebraic convex body with $0\in\int(K)$. Let $Z\subseteq\Ex_a(K)$ be an irreducible subvariety with $Z\cap\Ex(K)$ dense in $Z$ such that $\overline{Z}^*$ is an irreducible component of $\overline{\partial_aK^\dual}$. Then $Z$ is an irreducible component of one of $F_i$, which is defined by induction:
	\begin{align*}
	    F_0:=&\partial_aK,\  F_1:=\Sing(F_0),\\
	    F_i:=&\Sing(F_{i-1})\cup\bigcup_{j=0}^{i-2}S\left(F_j\setminus F_{j+1},\Reg(F_{i-1})\right),
	\end{align*}
    where $S(X,Y)$ is the set of points in which the pair $(X,Y)$ does not satisfy Whitney's condition~(a) (see Definition~\ref{def:whitcondi}~(a)).
\end{theorem}


{\ifcolor\color{magenta}\fi
The canonical Whitney stratification was introduced independently by Teissier \cite{teissier1982multiplicites} and Henry and Merle \cite{henry1983limites}.
}
The computation of Whitney stratifications for affine varieties is a challenging problem in computational geometry. 
{\ifcolor\color{magenta}\fi
Quantifier Elimination can be used to compute Whitney stratifications of complex algebraic sets in $\CC^n$ and semi-algebraic sets in $\RR^n$ \cite{Rannou1991, rannou98}.
}
The algorithm introduced in \cite{DJ2021} represents Whitney's condition~(b) using new variables and eliminates them to compute the stratifications. Another algorithm in \cite{HN2023} focuses on computing Whitney (b) irregular points through primary decomposition techniques. 
Since Theorem~1.3 only concerns Whitney's condition~(a), we present an algorithm, based on Teissier's criterion in \cite[Remark~4.11]{FloresTeissier2018}, to compute Whitney~(a) stratifications for equi-dimensional projective varieties using conormal spaces and prime decomposition.

\paragraph*{\textbf{Structure of the paper}} In Section~\ref{sec:dual-varieties}, we introduce some basic concepts in convex analysis and duality theory. In Section~\ref{sec:whitneyintro}, we introduce Whitney stratification of an algebraic variety. In Section~\ref{sec:mainthm}, we prove Theorem~\ref{thm:strata}. In Section~\ref{sec:conormal}, 
we give algorithms based on Teissier's criterion to compute Whitney stratifications using conormal spaces and prime decompositions. In Section \ref{sec:experiment}, 
we illustrate our main theorem with two examples: Xano and Teardrop.


\section{Duality of Varieties and Convex Sets}\label{sec:dual-varieties}
In this section, we introduce some basic concepts in convex algebraic geometry, based on \cite{MR1451876,rostalski10,rostalski2012chapter,Sinn2015}.

\subsection{Conormal Spaces, Dual Varieties and Duality of Convex Sets}

\begin{definition}
Let $X\subseteq \PP^n$ be 
{\ifcolor\color{teal}\fi a complex analytic space of pure dimension. }
The \emph{conormal space} $C(X)$ of $X$ is the closure of 
\begin{equation*}\label{defe:conormal}
    \{(x,l)\in\PP^n\times(\PP^n)^*\ |\ x\in\Reg(X),\ \<l,T_xX\>=0\},
\end{equation*}
where $\<l, T_xX \>=0$ means that the linear operator $l$ maps all vectors of $T_x X$ to $0$.
Here, we use $\Reg(X)$ to represent the regular (smooth) points in $X$ and $\Sing(X):=X\setminus\Reg(X)$.

The projection of the conormal space $C(X)$ onto the second factor $(\PP^n)^*$ is called the \emph{dual variety} of $X$, denoted by $X^*$. More precisely, $X^*$ is the closure of
    \begin{equation*}
        \{l\in(\PP^n)^*\ |\ \exists x\in\Reg(X),\ \<l,T_xX\>=0\}.
    \end{equation*}
\end{definition}
{\ifcolor\color{magenta}\fi
We have the biduality theorem for projective varieties \cite[Theorem 1.1]{Gelfand08}:  if $X\subseteq \PP^n$ is an irreducible projective variety, then $ (X^*)^* = X $.
}

\begin{definition}\label{def:convex}
	A subset $D\subseteq\RR^n$ is called \emph{convex} if for all pairs $x,y\in D$ and $0\le\lambda\le 1$, $\lambda x+(1-\lambda)y\in D$.
	
	If $D\subseteq\RR^n$ is a compact set of full dimension, that is, of dimension $n$, 
 we call $D$ a \emph{convex body}. If $0\in\int(D)$ the Euclidean interior of $D$, we define the \emph{dual convex body}
	\begin{equation*}
	D^\dual:=\{l\in (\RR^n)^*\ |\ \forall x\in D,\ \<l,x\>\ge -1\}.
	\end{equation*}
	
	For a point $x\in\partial D$, a hyperplane $l\in \partial D^\dual$ satisfying $\<l,x\>=-1$ is called a \emph{supporting hyperplane} of $x$.
\end{definition}

For a semi-algebraic convex body $D$ and a point $x$ on the boundary, there is at least one hyperplane $l$ supporting $x$ by {\cite[Theorem~18.7]{MR1451876}}.

\begin{definition}
	A subset $C\subseteq\RR^{n+1}$ is called a \emph{cone} if for all $\lambda\ge 0$ and $x\in C$, $\lambda x\in C$.	A closed cone is called \emph{pointed} if it does not contain a line, namely $C\cap (-C)=\{0\}$. 
	
	Let $C\subseteq\RR^{n+1}$ be a full-dimensional closed pointed convex cone. We define the \emph{dual convex cone}
	\begin{equation*}
	    C^\vee:=\{l\in(\RR^{n+1})^*\ |\ \forall x\in C,\ \<l,x\>\ge 0\}.
	\end{equation*}
\end{definition}

The following remark shows that convex bodies and convex cones can be converted into each other.

\begin{remark}\label{rem:pullup}
	Let $D\subseteq\RR^n$ be a convex set. Define $\co(D):=\{(\lambda,\lambda x)\in \RR^{n+1}\ |\ \lambda\ge0,\ x\in D\}$, which is a pointed convex cone corresponded to $D$. One can verify that if $0\in\int(D)$, then we have
	\begin{equation*}
	\co(D^\dual)=\co(D)^\vee.
	\end{equation*}
	In fact, for $\lambda,\mu>0$, $\<(\lambda,\lambda l),(\mu,\mu x)\>\ge 0 \iff \lambda\mu(1+\<l,x\>) \ge 0\iff \<l,x\>\ge -1$.
\end{remark}

If $D\subseteq\RR^n$ is a convex body with $0\in\int(D)$, then we have
	\begin{equation*}
    	(D^\dual)^\dual=D.
	\end{equation*}
	If $C\subseteq\RR^{n+1}$ is a full-dimensional closed pointed convex cone, then we have
	\begin{equation*}
    	(C^\vee)^\vee=C.
	\end{equation*}

\subsection{Boundary of Convex Semi-Algebraic Sets}

In this subsection, we introduce the notations of algebraic boundaries, normal cones, and some of their properties.

\begin{definition}
	Let $S\subseteq\RR^n$ be a semi-algebraic set and $\partial S$ be the Euclidean boundary of $S$. The \emph{algebraic boundary} of $S$, denoted $\partial_a S$, is the Zariski closure of $\partial S$.
\end{definition}

\begin{proposition}{\upshape\cite[Corollary~2.1]{Sinn2015}}\label{prop:dim}
	Let $K\subseteq\RR^n$ be a nonempty semi-algebraic convex body. Then, each irreducible component of $\partial_aK$ has codimension one in $\AA^n$.
\end{proposition}

\begin{definition}\label{projclos}
	Let $X\subseteq\AA^n$ be an affine variety. We define its \emph{projective closure} $\overline{X}\in \PP^n$ as the Zariski closure of the image of $X$ under the embedding 
	\begin{equation}
	\AA^n\hookrightarrow\PP^n,\ (x_1,\ldots,x_n)\longmapsto(1:x_1:\ldots:x_n).
	\end{equation}	
\end{definition}
{\ifcolor\color{magenta}\fi
For convenience, we will briefly use the term ``the dual $X^*$ to an affine variety $X$'' instead of ``the dual $\overline{X}^*$ to the projective closure $\overline{X}$ of an affine variety $X$''.
}

\begin{proposition}\label{lem:tangent}{\upshape\cite[Proposition~2.12]{Sinn2015}}
	 Let $K\subseteq\RR^n$ be a semi-algebraic convex body with $0\in\int(K)$, and let $M:=\partial K\cap\Reg(\partial_aK)$ be the smooth points on the boundary. Then, for any $x\in M$, there is exactly one hyperplane in $l\in K^\dual$ supporting $x$. Moreover, we have $\<l,T_xM\>=0$.
\end{proposition}

\begin{definition}\label{def:exa}
	Let $D\subseteq\RR^n$ be a convex set. A point $x\in D$ is called an \emph{extreme point} if for any $y,z\in D$ and $0<\lambda<1$ with $x=\lambda y+(1-\lambda)z$, it holds that $x=y=z$. The set of extreme points is denoted by $\Ex(D)$. We denote the Zariski closure of $\Ex(D)$ in $\AA^n$ by $\Ex_a(D)$.
	
	Let $C\subseteq\RR^{n+1}$ be a pointed convex cone. A ray in $C$ is denoted by $\RR_+x$, and $\RR_+x=\RR_+y$ if and only if $x=\lambda y$ for some $\lambda>0$. A ray $\RR_+x\subseteq C$ is called an \emph{extreme ray} if for any $y,z\in C$ and $0<\lambda<1$ with $x=\lambda y+(1-\lambda)z$, it holds that $\RR_+x=\RR_+y=\RR_+z$. 
 The set of extreme rays is denoted by $\Exr(C)$. We denote the Zariski closure of the union of all the extreme rays by $\Exr_a(C)$.
\end{definition}




\begin{lemma}\label{cor:nbhd}
	Let $K\subseteq\RR^n$ be a semi-algebraic convex body with $0\in\int(K)$. Let $l\in\Ex(K^\dual)$ be a supporting hyperplane to a point $x\in\Ex(K)$. 
 Let $\{l_i\}\subseteq\partial K^\dual$ be an infinite sequence converging to $l$ and $x_i\in \partial_a K$ be a point supported by $l_i$.
 Then the sequence $\{x_i\}$ converges to $x$.
\end{lemma}
\begin{proof}
By Definition \ref{def:convex}, we have $\<l_i,x_i\>=-1$ and $\<l,x\>=-1$. If $\{x_i\}$ does not converge to $x$, by the assumption that $K$ is compact, we can assume there is a subsequence of $\{x_i\}$, denoted by $\{x_{k_i}\}$, converging to $x'\in K$. So $\<l,x'\>=\lim_{i\to\infty}\<l_{k_i},x_{k_i}\>=-1$. This shows that  $l$ is a supporting hyperplane to $x'$. 
{As $$\<l,x\>=\<l,x’\>=-1,$$ we conclude that $x$ and $x’$ are in the same face. By $x\in \Ex(K)$, we have $x=x’$. Hence, the sequence $\{x_i\}$ converges to $x$.}
\end{proof}

Let $D\subseteq\RR^n$ be a convex set, $x\in\partial D$, we define 
\begin{equation*}
    N_D(x)=\{l\in(\RR^n)^*\ |\ \forall y\in D,\ \<l,y-x\>\ge0\}
\end{equation*}
as the \emph{normal cone} of $x$.

\section{Stratification on Varieties}\label{sec:whitneyintro}
The idea of stratification in algebraic geometry comes from topology, where an important approach is to divide the topological space into smaller and simpler parts. In \cite{MR0095844}, Whitney proved that an algebraic variety can be decomposed into several smooth submanifolds. In \cite{MR0188486,MR0192520}, Whitney refined the stratification of iterated singular loci into a stratification satisfying Whitney's condition (a) or (b). This section briefly introduces stratifications on varieties based on \cite{MR1234122,MR1226270,rannou98,MR4261554}.  

\begin{definition}
	Let $F$ be an algebraic set of $\CC^n$ or $\RR^n$. A \emph{stratification} of $F$ is a family of subsets $\{F_\alpha\}$ such that
	\begin{enumerate}
		\item $\{F_\alpha\}$ are pairwise disjoint;
		\item For each pair $F_\alpha$ and $F_\beta$, either $F_\beta\subseteq\cls(F_\alpha)$ or $\cls(F_\alpha)\cap F_\beta=\varnothing$, here $\cls(F_\alpha)$ means the Euclidean closure of $F_\alpha$;
		\item Each $F_\alpha$ is smooth. 
	\end{enumerate} 
	We call each $F_\alpha$ a \emph{stratum}.
\end{definition}

\begin{definition}[Whitney's Conditions]{\upshape\cite[Definition 9.7.1]{MR1659509}}\label{def:whitcondi}
	Let $X,Y$ be two strata of a real algebraic set $F$.
\begin{enumerate}[(i)]
\item 
  The pair $(X,Y)$ satisfies \emph{Whitney's Condition (a)} at a point $y\in Y$ if for every sequence $\{x_i\}_{i\in\NN}$ with $x_i\in X$ and $\lim_{i\to\infty} x_i=y$, it holds that $T_yY\subseteq\lim_{i\to\infty} T_{x_i}X$ if the latter limit exists. 
\item 
  The pair $(X,Y)$ satisfies \emph{Whitney's Condition (b)} at a point $y\in Y$ if for every sequences $\{x_i\}_{i\in\NN}\subset X$ and $\{y_i\}_{i\in\NN}\subset Y$ with $\lim_{i\to\infty} x_i=\lim_{i\to\infty} y_i=y$, it holds that $\lim_{i\to\infty} \RR(x_i-y_i)\subseteq\lim_{i\to\infty} T_{x_i}X$ if these two limits both exist.
	\end{enumerate}
\end{definition}

In fact, Whitney's condition~(b) implies condition~(a), see e.g. \cite[Proposition 2.4]{mather2012notes}.
We use the terminology from \cite{DJ2021}: The pair $(X,Y)$ is said to be \emph{Whitney (a) regular} (resp. Whitney regular) if $(X,Y)$ satisfies Whitney's Condition (a) (resp. Whitney's Condition (b)) at every point of $Y$.
The limit in Definition \ref{def:whitcondi} is taken with respect to the Euclidean topology of Grassmannians $\GG_n^k(\RR)$, which is described in detail in \cite[Section 9.7]{MR1659509}.

{\ifcolor\color{magenta}\fi
Teissier \cite{teissier1982multiplicites} and Henry and Merle \cite{henry1983limites} independently introduced the canonical Whitney stratification. Rannou and Mostowski \cite{Rannou1991} gave an algorithm based on quantifier elimination to calculate it.
}

\begin{definition}\label{def:filtration}
	Let $F\subseteq\RR^n$ be an algebraic set. We define 
  the \emph{Whitney (a) stratification} of $F$
	\begin{equation*}
	F=F_0\supsetneq F_1\supsetneq F_2\supsetneq\cdots\supsetneq F_r\supsetneq F_{r+1}=\varnothing
	\end{equation*}
	where 
\begin{enumerate}
    \item $F_1:=\Sing(F_0)$. 
	\item For $i \geq 2$, let 
	\begin{equation*}
	F_i:=\Sing(F_{i-1})\cup\bigcup_{j=0}^{i-2}S\left(F_j\setminus F_{j+1},\Reg(F_{i-1})\right)
 \end{equation*}
 where $S(X,Y)$ is the set of points in which the pair $(X,Y)$ does not satisfy Whitney's condition~(a).
 \item  $F_r\neq\varnothing$.
\end{enumerate} 	
\end{definition}

Let $F_i^j$ be the $j$-th connected component of $F_i\setminus F_{i+1}$. One can verify that by this definition, $\{F_i^j\}$ 
forms a stratification for which Whitney's condition (a) is satisfied for every pair of strata.
Note that the stratification 
{\ifcolor\color{magenta}\fi in Definiton~\ref{def:filtration} }
is the minimum of the ones that satisfy Whitney's condition (a). 
In practice, for two strata $X = F_{j_1}\setminus F_{j_1+1}$ and $Y=F_{j_2}\setminus F_{j_2+1}$ with $Y\subsetneq X$, we compute the Zariski closure of the set of the points where $(X,Y)$ does not satisfy Whitney's condition~(a),  denoted as $\text{Zar}(S(X,Y))$. Then $(X, Y\setminus \text{Zar}(S(X,Y))$ is
{\ifcolor\color{magenta}\fi also }
Whitney~(a) regular. 

	By {\upshape \cite[Section~3]{Rannou1991}} and {\upshape \cite[Theorem~9.7.5]{MR1659509}}, the set $S(X,Y)$ 
 has dimension strictly less than $Y$. Therefore, $\dim(F_{i+1} < \dim(F_i)$ for all $i\ge 0$, so the sequence $(F_0, F_1,\ldots, F_{r})$ has length no more than $\dim(F_0)+1$.

\section{Stratification of Algebraic Boundaries of Convex Semi-Algebraic Sets}\label{sec:mainthm}

A point $x$ on a real projective $X$ is called \emph{central} if it is the limit of a sequence of regular real points of $X$ \cite[Definition~7.6.3]{MR1659509}. Let $K\subseteq\RR^n$ be a semi-algebraic convex body with $0\in\int(K)$. For an extreme point $x\in\Ex(K)$, the embedding image $(1:x)\in\PP^n$ is a central point of $\overline{Y}^*$ which is the dual variety of an irreducible component $\overline{Y}$ of $\overline{\partial_aK^\dual}$ \cite[Corollary~3.14]{Sinn2015}. 


\begin{lemma}\label{lem:span}
Let $K\subseteq\RR^n$ be a semi-algebraic convex body with $0\in\int(K)$. Let $Z\subseteq\Ex_a(K)$ be an irreducible subvariety such that $\overline{Z}^*$ is an irreducible component of $\overline{\partial_aK^\dual}$, and $Z\cap \Ex(K)$ is dense in $Z$. Let $x\in Z$ be a general point and $N_K(x)$ be its normal cone. Then we have $\mathrm{span}(N_K(x))=(T_x Z)^\bot$.
\end{lemma}

\begin{proof}
Let $x\in Z\cap\Ex(K)$. Considering $x$ as a supporting hyperplane of $K^\dual$, it has a nonempty intersection with $\Ex(K^\dual)$. Therefore, there exists an $l\in\Ex(K^\dual)$ with $\RR_+l\in \Exr(N_K(x))$ and $l$ is a supporting hyperplane to $x$, . By \cite[Corrollary 3.14]{Sinn2015}, there exists an irreducible component $X$ of $\partial_aK$ such that $l$ is a central point of $X^*$. Namely, there exists a sequence $\{l_i\}\subseteq \Reg(X^*)$ converging to $l$. Let $x_i\in X$ be the point supported by $l_i$, then $l_i=(T_{x_i}X)^*$. By \cite[Theorem 9.7.5]{MR1659509}, Whitney's condition (a) is satisfied for the pair $(\Reg(X),Z)$ for a general point $x\in Z$, i.e. $T_xZ\subseteq\lim_{i\to\infty} T_{x_i}X$. Hence we have $l=\lim_{i\to\infty} l_i\in(T_xZ)^*$. For any ray $\RR_+l\in \Exr(N_K(x))$, 
\[
    \RR_+l\subseteq\RR\cdot(T_xZ)^*=(T_xZ)^\bot,
\]
Then, we have $N_K(x)\subseteq(T_xZ)^\bot$.

As $\overline{Z}^*$ is an irreducible component of $\overline{\partial_aK^\dual}$, according to \cite[Corollary 3.9]{Sinn2015}, we get the following equations:
\begin{equation*}
    \dim(N_K(x))=\codim(Z)=\codim(T_xZ)=\dim((T_xZ)^\bot).
\end{equation*}
Because $N_K(x)\subseteq(T_xZ)^\bot$ and $\dim(N_K(x))=\dim((T_xZ)^\bot)$, we conclude that
 $ \mathrm{span}(N_K(x))=(T_xZ)^\bot$.
\end{proof}


\begin{theorem}\label{thm:strata}
	Let $K\subseteq\RR^n$ be a semi-algebraic convex body with $0\in\int(K)$. Let $Z\subseteq\Ex_a(K)$ be an irreducible subvariety with $Z\cap\Ex(K)$ dense in $Z$ such that $\overline{Z}^*$ is an irreducible component of $\overline{\partial_aK^\dual}$. Then $Z$ is an irreducible component of one of $F_i$, which is defined by induction:
	\begin{align*}
	    F_0:=&\partial_aK,\  F_1=\Sing(F_0),\\
	    F_i:=&\Sing(F_{i-1})\cup\bigcup_{j=0}^{i-2}S\left(F_j\setminus F_{j+1},\Reg(F_{i-1})\right),
	\end{align*}
	 where $S(X,Y)$ is the set of points in which the pair $(X,Y)$ does not satisfy Whitney's condition~(a).
\end{theorem}
The proof below follows the main idea in the proof of \cite[Theorem 3.16]{Sinn2015}. 
\begin{proof}
	Let $\{F_i\}$ be a stratification of $\partial_aK$ satisfying Whitney's condition (a). Assume that 
 \begin{equation}\label{assmp}
 Z\subset F_k,  ~~~  Z\not\subset F_{k+1}.
 \end{equation}
 Let $Y$ be the irreducible component of $F_k$ that contains $Z$. We assume that $Z\not=Y$, i.e.
	$
	Z\subsetneq Y\subseteq F_k.$
 
	Let $x\in Z\cap Ex(K)$ be an extreme point of $K$. For a general point $x\in\Reg(Y)$, we have $T_xZ\subsetneq T_xY$. By Lemma \ref{lem:span}, 
   we have 
   \[(T_xY)^\bot \subsetneq (T_xZ)^\bot=\mathrm{span}(N_K(x)). \]
   Hence, there exists $\RR_+l\in N_K(x)$ with $l\in\Ex(K^\dual)$ such that 
   \[\<l,T_xY\>\neq 0.\]
   Since $l$ is an extreme point of $K^\dual$, by \cite[Corollary~3.14]{Sinn2015}, there exists an irreducible component of $\partial_aK$, denoted by $X$, satisfying that $l$ is a central point of its dual $X^*$, i.e. there exists a sequence $l_i\to l$ with $l_i\in \Reg(X^*)$. By Lemma \ref{cor:nbhd}, there exists a sequence $x_i\to x$ with $x_i\in\Reg(X)$, and each $l_i$ is the supporting hyperplane to $x_i$. Then, by Proposition \ref{lem:tangent}, we have 
   $\langle l_i,T_{x_i}X\rangle=0$. Therefore, 
	\begin{equation*}
	    \langle l,\lim_{i\to\infty} T_{x_i}X\rangle=\lim_{i\to\infty} \langle l_i,T_{x_i}X\rangle=0.
	\end{equation*}
	Because the Grassmannian $\GG_n^k$ is compact, without loss of generality, we can assume that $\lim_{i\to\infty} T_{x_i}X$ exists.   Since $\langle l, T_xY\rangle\neq0$, we have 
 \[T_xY\not\subset\lim_{i\to\infty} T_{x_i}X.\] 
{This implies that for any $x\in Z\cap\Ex(K)$, we have}
	\begin{equation*}
	x\in S(\Reg(X),\Reg(Y)).
	\end{equation*}
Since $Z\cap\Ex(K)$ is dense in $Z$ and $S(\Reg(X),\Reg(Y))$ is closed,  we conclude that 
	\begin{equation*}
		Z\subset \Sing(Y)\cup\left(\bigcup_{X\in\mathrm{irr}(\partial_aK)} S\left(\Reg\left(X\right),\Reg(Y)\right)\right)\subset F_{k+1},    
	\end{equation*}
which is a contradiction to the assumption (\ref{assmp}).
\end{proof}

\section{Computing Whitney~(a) stratifications via conormal spaces}
\label{sec:conormal}
In this section, we fix $\PP^n$ as the $n$-dimensional projective space over the complex number field $\CC$.
{\ifcolor\color{teal}\fi
We show that the algorithm originally proposed to compute Whitney (b) stratification by Helmer and Nanda in \cite{HN2022}, which was subsequently found to contain an error in regards to the output satisfying condition (b) and corrected by Helmer and Nanda in \cite{HN2023}, does in fact correctly compute a Whitney (a) stratification. Along with a proof of correctness relative to condition (a) we present a slightly altered version of the algorithm of \cite{HN2022} which is better suited to our application. The algorithm is based on an algebraic description of Whitney's conditions given in \cite{FloresTeissier2018}.}
\begin{definition}
\label{def:conormal}
{\ifcolor\color{teal} Let $X\subseteq \PP^n$ be a complex analytic space of pure dimension.\fi} 
Let $C(X)$ be the conormal space of $X$.
The {\upshape conormal map} $\kappa_X$ is defined as
\begin{equation}
\begin{tabular}{c@{\hspace*{0.3em}}c@{\hspace*{0.3em}}@{}c@{\hspace*{0.3em}}c}
$\kappa_X$: & $\cx$     & $\to$     &$X$\\
            & $(x,l)$ & $\mapsto$ &$x$.
\end{tabular}
\end{equation}
\end{definition}

{\ifcolor\color{teal}\fi
The conormal space $C(X)$ is closed and reduced (see e.g. \cite[Proposition 4.1]{teissier1982multiplicites} or \cite[Proposition 2.9]{FloresTeissier2018}),
{\ifcolor\color{magenta}\fi
therefore, the sheaf of ideals $\icx$ defining $C(X)$ corresponds to the vanishing ideal of $C(X)$.
}
The computation of $\icx$ is also given in the proof \cite[Proposition 4.1]{teissier1982multiplicites}.

In this section, all ideals are homogeneous since we only consider projective varieties. Although ``homogeneous ideal'' is more precise in this context, we will simply use ``ideal'' for brevity.
}

\begin{lemma}
\label{lem:teissier}
{\upshape \cite[Remark~4.11]{FloresTeissier2018}}
Let $X$ be an irreducible variety in $\PP^n$ and let $Y$ be a smooth
irreducible subvariety of $X$. Let $\icxy=\icx+\icy$ be the sheaf of ideals defining $C(X)\cap C(Y) $, and $\ikxy=\icx+\iy$ be the sheaf of ideals defining $\kappa^{-1}_X(Y)$. 
Then $(\Reg(X),Y) $ satisfies Whitney's condition~(a) if and only if
\begin{equation}
\label{eq:whita}
\ikxy \subset \icxy \subset \sqrt{\ikxy},
\end{equation}
and $(\Reg(X),Y) $ satisfies
Whitney's condition~(b) if and only if
\begin{equation}
\label{eq:whitb}
\ikxy \subset \icxy \subset \overline{\ikxy},
\end{equation}
where $\overline{\ikxy} $ is the integral closure of $\ikxy$.
\end{lemma}

In Lemma~\ref{lem:teissier}, $Y$ is closed as it is a subvariety of $X$.
We observe that the closedness of $Y$ is necessary: 
{\ifcolor\color{magenta}\fi
if $Y$ is not closed, then the first inclusion in \eqref{eq:whita} and \eqref{eq:whitb} may not hold as shown in Example~\ref{ex:whit_cusp}, and the second inclusion in \eqref{eq:whitb} may not hold either 
\cite[Example 3.2]{LLYZ24}.
}

\begin{example}
\label{ex:whit_cusp}
Let $ f = x^2  t^2-y^2 z^2+z^3 t \in \CC[x,y,z,t] $ 
be a homogeneous polynomial. Let 
\begin{equation*}
    X=\VV_{\CC}(f)\subset \PP^3 \text{ and } \ Y = \VV_{\CC}(x,z)\setminus \VV_{\CC}(y)\subset \sing(X).
\end{equation*}
Although $(\Reg(X), Y)$ is Whitney~(a) regular, $\ikxy\not\subset\icxy$. 
Here the ideal  $\ikxy$ is computed by \cite[Eq.~$(5)$]{HN2023}, which is a result of the remarks below \cite[Theorem 3.12]{Mumford15}.
\end{example}


Nonetheless, we show that the second inclusion of \eqref{eq:whita} 
still holds even if $Y$ is 
{\ifcolor\color{magenta}\fi Zariski open.}

\begin{corollary}
\label{cor:whita}
Let $X$ be an irreducible variety in $\PP^n$ and $Y$ be an irreducible (not necessarily smooth) subvariety of $X$.
Let $\xzero = X\setminus \sing(X)$, and $\yzero\subset Y\setminus\sing(Y)$ be a Zariski open subset of $Y$.  Let $\icxyz$ be the sheaf of ideals defining $C(X)\cap C(\yzero)$ and $\ikxyz$ be the sheaf of ideals defining $\kappa^{-1}_X(\yzero)$.
Then $(\xzero, \yzero)$ satisfies Whitney's condition~(a) if and only if 
\begin{equation} 
\label{eq:whita2} 
\sqrt{\icxyz} \subset \sqrt{\ikxyz}.
\end{equation} 
\end{corollary}

\begin{proof}
By Hilbert Nullstellensatz, the condition~\eqref{eq:whita2} is equivalent to 
\begin{equation}
\label{eq:whita_bridge}
\text{Zariski closure of}\left(\kappa_X^{-1}(\yzero)\right)
\subset C(X) \setcap C(\yzero).
\end{equation}
Since $C(X) \setcap C(\yzero)$ is Zariski closed and
\begin{equation*}
\kappa_X^{-1}(\yzero) = C(X) \setcap \big( \yzero \times (\PP^n)^*\big),
\end{equation*}
the inclusion
in \eqref{eq:whita_bridge} is equivalent to
\begin{equation}
\label{eq:whita3}
C(X) \setcap \big( \yzero\times (\PP^n)^*\big) \subset C(X) \setcap C(\yzero).
\end{equation}
In the following, we prove that
$(\xzero,\yzero)$ is Whitney~(a) regular if and only if \eqref{eq:whita3}
holds.

First, we show that \eqref{eq:whita3} leads to $(\xzero,\yzero)$
being Whitney~(a) regular. Let $y\in \yzero$, and let $(x^{(k)})$ be a sequence in $\xzero$
such that $x^{(k)}\to y$ and $T_{x^{(k)}}\xzero \to T$ as $k\to \infty$.
Then \eqref{eq:whita3} implies that
\begin{equation*}
   \forall l\in (\PP^n)^*: \<l, T\>=0 \Longrightarrow \<l, T_y\yzero\> = 0, 
\end{equation*}
which concludes $T_y\yzero \subset T$.

Conversely, we show that \eqref{eq:whita3} holds if $(\xzero, \yzero)$ is Whitney~(a) regular.
Let $(y,l)$ be a point in $C(X)\setcap (\yzero\times (\PP^n)^*)$. 
Since $C(X)$ is the closure of the following set:
\begin{equation*}
    \{(x,l)\in\PP^n\times(\PP^n)^*\ |\ x\in \xzero,\ \<l,T_xX\>=0\},
\end{equation*}
there exists a sequence $(x^{(k)}, l^{(k)})\in \xzero\times (\PP^n)^*$ such that
\begin{equation*}
    (x^{(k)}, l^{(k)})\to (y,l) \text{ as } k\to \infty.
\end{equation*}
Then the compactness of $\PP^n$ implies that
$(T_{x^{(k)}}\xzero)$ has a convergent subsequence $(T_{x^{(k_i)}}\xzero)$, and let
$T\subset \PP^n$ be the limit of $T_{x^{(k_i)}}\xzero$.
Consequently, $\< l, T\> = 0$ due to the continuity of the inner product.
As $(\xzero, \yzero)$ is Whitney~(a) regular, we have
$T_y \yzero\subset T$, which leads to $\<l, T_y\yzero \>  = 0$, and hence 
$(y,l)\in C(\yzero)$.
\end{proof}

With the notation in Corollary~\ref{cor:whita}, if
\begin{equation}
\label{eq:whita_subsetneq}
\ikxyz\not\subset \sqrt{\icxyz}  \subsetneq \sqrt{\ikxyz},
\end{equation} 
then there is some subvariety $W$ of $Y$ such that $(\xzero, \yzero\cup W)$
is also Whitney~(a) regular. 
In other words, if we are trying to stratify $X$ and we get a relation
as \eqref{eq:whita_subsetneq}, then we  know that
we have removed too many points from $Y$ while getting $\yzero$. 

\begin{lemma}
    \label{lem:whita_correctness}
    Let $X\subset \PP^n$ be an irreducible variety and $Y\subsetneq X$ be an irreducible subvariety of $X$.
    Let $Q=\ikxy \colon \icxy^{\infty}$ and $\yzero=\Reg(Y)\setminus\kappa_X(\VV_\CC(Q))$. Then
    \begin{equation}
        \label{eq:nwa}
        \sqrt{\icxyz} \subset \sqrt{\ikxyz},
    \end{equation}
    where $\kappa_X\colon C(X)\to X$ is the conormal map from $C(X)$, the conormal space of $X$, to $X$.
    Moreover, $(\Reg(X), \yzero)$ satisfies Whitney's condition~(a).
\end{lemma}
\begin{proof}
By Hilbert Nullstellensatz, it is equivalent to prove 
\begin{equation}
\label{eq:kxz_sub_cxz}
    \kappa_X^{-1}(\yzero)\subset C(X)\cap C(\yzero).
\end{equation}
According to \cite[Lemma~4.2]{HN2022}, $\yzero$ is Zariski dense in $Y$. Consequently, $C(\yzero)=C(Y)$ by \cite[Proposition~2.2]{HN2022}.
Then the inclusion \eqref{eq:kxz_sub_cxz} follows from
\begin{equation*}
    \kappa_X^{-1}(\yzero)
    \subset \kappa_X^{-1}(Y)\setminus\VV_\CC(Q)
    \subset C(X)\cap C(Y)
    =C(X)\cap C(\yzero),
\end{equation*}
where the second inclusion holds because $\VV_\CC(Q)$ is equal to the Zariski closure of 
$\big(\kappa_X^{-1}(Y)\setminus C(X)\cap C(Y)\big)$ (see e.g. \cite[Chapter~4, $\S4$, Theorem~10~$(iii)$]{CLO2015ideals}). We conclude that $(\Reg(X), \yzero)$ satisfies Whitney's condition~(a) by Corollary~\ref{cor:whita}.
\end{proof}

Given two irreducible projective varieties $X$ and $Y$,
where $Y$ is contained in the singular locus of $X$,  we
present Algorithm~\ref{alg:notwhit}, based on Corollary~\ref{cor:whita}
Lemma~\ref{lem:whita_correctness}, to compute a subvariety
$W$ of $\Reg(Y)$ such that $(\Reg(X),\Reg(Y)\setminus W)$
satisfies Whitney's condition~(a), i.e., $W=S(\Reg(X),\Reg(Y))$ as in Definition~\ref{def:filtration}. Next, we utilize 
{\ifcolor\color{teal}\fi Algorithm~\ref{alg:notwhit}} 
as a subroutine to compute Whitney~(a) stratifications of 
 projective varieties. 

{\ifcolor\color{teal}\fi The algorithms in \cite{DJ2021, HN2023} are for computing Whitney (b) irregular points of $(X,Y)$, which can also be used for computing Whitney (a) irregular points because Whitney's condition (b) implies condition (a). However, with the sufficient and necessary condition in Corollary~\ref{cor:whita} and the fact that Whitney's condition (a) is weaker than condition (b), our Algorithm~\ref{alg:notwhit} is simpler and should run faster than the algorithms in \cite{DJ2021, HN2023}}. 

\begin{algorithm}
\caption{Compute Whitney~(a) irregular points}\label{alg:notwhit}
\SetAlgoLined

\KwIn{
Two prime homogeneous ideals $I, J\subset\CC[x_0,x_1,\ldots,x_n]$ 
such that $Y:=\VV_\CC(J)$ is contained in the singular locus of $X:=\VV_\CC(I)$.}
\KwOut{
A list of homogeneous prime ideals $L_a=\{P_1,\ldots, P_k\}$ such that
$W=\bigcup_{i=1}^k\VV_\CC(P_i)$ contains the points in $\Reg(Y)$ where $(\Reg(X), \Reg(Y))$ does not satisfy Whitney's condition~(a).}
\begin{enumerate}[1.]
\item
Compute the vanishing ideals $\icx$ and $\icy$
of $C(X)$ and $C(Y)$ respectively.
\item
Let $\icxy=\icx + \icy$ 
and let $\ikxy = \icx+J$ 
by considering $J$ as an ideal in the ring of $\icx$.
Compute the saturation 
$Q_W = \sqrt{\ikxy}:\icxy^\infty$.
\item Compute the minimal primes $\{Q_1,\ldots, Q_\ell\}$ of $Q_W$.
\item For $i=1,\ldots,\ell$, compute the elimination ideal $P_i=Q_i\cap \CC[x_0,x_1,\ldots, x_n]$, and put $P_i$ into the list $L_a$.
\item
Remove the redundant ideals from $L_a$ so that 
{\ifcolor\color{magenta}\fi the}
ideals in $L_a$ are distinct and the corresponding varieties are not empty sets; return $L_a$.
\end{enumerate}
\end{algorithm}



 
\begin{algorithm}
\caption{Compute Whitney~(a) stratification of 
a projective variety.}
\label{alg:whitneya}
\SetAlgoLined

\KwIn{
A homogeneous 
radical ideal $I\subset\CC[x_0,x_1,\ldots,x_n]$ with $X:=\VV_\CC(I)$.
}
\KwOut{
A list $L=\{L_0,L_1,\ldots, L_r\}$ such that
\begin{itemize}
\item $L_i$ is  a list of 
prime ideals $\{P_{i,1},\ldots,P_{i,k_i}\}$ for $i=0,1,\ldots,r$;
\item $\bigcap_s P_{i,s}\subset \bigcap_t P_{i+1,t}$ for all $i=0,\ldots,r-1$;
\item $X=\bigsqcup_{i=0}^{r}\bigsqcup_{j=1}^{k_i}\Reg(X_{i,j})$
is a Whitney~(a) stratification of the variety $X$, where
$X_{i,j}=\VV_\CC(P_{i,j})$ and $P_{i,j}\in L_i$.
\end{itemize}
}
\SetAlgoLined
\begin{enumerate}[1.]
\item
Compute prime decomposition $I=P_{0,1}\cap\cdots\cap P_{0,k_0}$ and
set $L_0=\{P_{0,1},\ldots,P_{0,k_0}\}$;
\item Compute the  vanishing ideal of $\Sing(X)$ 
using the Jacobian criterion, 
denoted $I_\text{sing}$. Compute prime decomposition $I_\text{sing} =P_{1,1}\cap\cdots\cap P_{1,k_1}$ and set $L_{1}=\{P_{1,1},\ldots,P_{1,k_1}\}$;
\item 
For $1\leq j\leq \dim X$
\begin{enumerate}[(a)]
\item compute the minimal primes of the vanishing ideal of the singular locus of the variety $\bigcup_{P_{j,\mu}\in L_j}\VV_\CC(P_{j,\mu})$; add the computed prime ideals to the list $L_{j+1}$. 
\item for $i=0,\ldots,j-1$ and for all  $P_{i,\ell_i}\in L_{i}$ and $P_{j,\ell_j}\in L_{j}$ with $P_{i,\ell_i}\subsetneq P_{j,\ell_j}$, call Algorithm~\ref{alg:notwhit}
on $(P_{i,\ell_i}, P_{j,\ell_j})$, add the output to the list $L_{j+1}$; 
{\ifcolor\color{magenta}\fi remove the redundant ideals from $L_{j+1}$ such that the ideals in $L_{j+1}$ are distinct and the corresponding varieties are not empty.}
\end{enumerate}
\item Return $L$.
\end{enumerate}
\end{algorithm}

\begin{theorem}
Algorithm~\ref{alg:whitneya} terminates in finite steps 
and outputs a list of prime ideals satisfying the output
specifications.
\end{theorem}
\begin{proof}
Algorithm~\ref{alg:whitneya} terminates in finite steps because every ideal concerned in the algorithm is finite-dimensional and has finitely many minimal primes.

The correctness of Algorithm~\ref{alg:whitneya} follows
from the correctness of Algorithm~\ref{alg:notwhit}, which
is established by Lemma~\ref{lem:whita_correctness}.
\end{proof}

{\ifcolor\color{magenta}\fi
The main cost of our algorithm is the computation of the vanishing ideals of conormal spaces and prime decomposition of ideals; we refer to \cite[Theorem~8.4]{HN2022} for the detailed complexity estimate for the algorithm 
``WhitStrat''. 
}

For a real algebraic set $X\subset\mathbb{R}^n$, let $X(\mathbb{C})$ be the complexification of $X$.  According to \cite[Theorem~3.3]{HN23Real}, one can 
obtain a Whitney stratification of $X$ by intersecting each stratum $(F_i)$ of the Whitney stratification of $X(\mathbb{C})$  with $\mathbb{R}^n$.

  

\section{Examples}\label{sec:experiment}
{\ifcolor\color{magenta}\fi We compute }
two examples to support Theorem~\ref{thm:strata}. Each example contains 
{\ifcolor\color{teal}\fi an extreme point }
whose dual variety is an irreducible component of the algebraic boundary of the dual convex body, and this point can be discovered by Whitney (a) stratification but can not be identified by computing iterated singular loci. 
{\ifcolor\color{magenta}\fi We utilized  Macaulay2 and Maple2023 for the computation.}
\subsection{Xano}

\begin{example}
\label{ex:xano}
 The variety $\VV({x^4+z^3-yz^2})$ is called \emph{Xano} from the famous illustration of Hauser \cite{Hauserweb}. One can verify that Xano is locally convex by computing the  Hessian matrix of $y=z+\frac{x^4}{z^2}$: 
\begin{equation*}
    H_{x,z}(y)=\left(\begin{array}{cc}
       \frac{12x^2}{z^2}  &  -\frac{8x^3}{z^3}\\
       -\frac{8x^3}{z^3}  & \frac{6x^4}{z^4}
    \end{array}
    \right),
\end{equation*}
which is a positive semidefinite matrix. After changing the coordinates of  Xano, we have a convex semi-algebraic set, which contains the origin as an interior point,  defined by
$$
K=\{(x,y,z) \in \RR^3 ~|~ f\le 0,\ y\le 1,\ -1\le z\},
$$ where $f = (x^4+(z+1)^3-(y+2)(z+1)^2)\in \QQ[x,y,z]$. 
\begin{figure}[tbhp] 
  \centering 
  \includegraphics[width=0.46\textwidth]{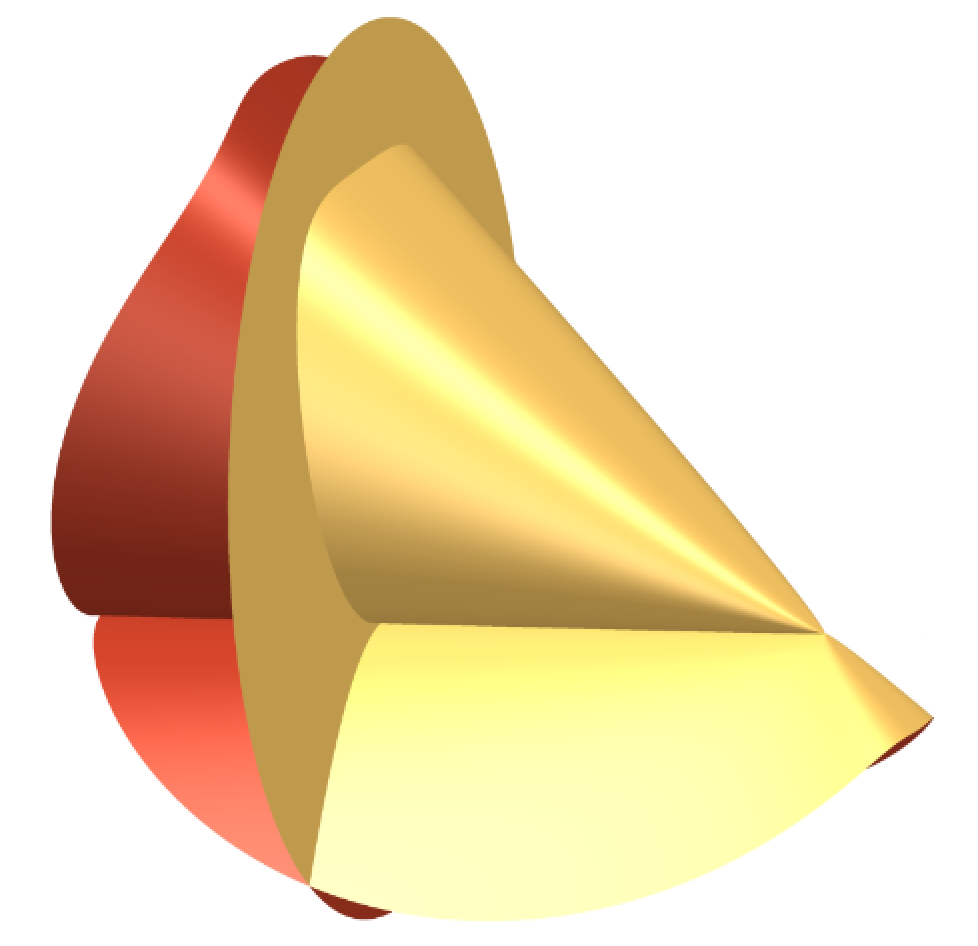}
  \includegraphics[width=0.46\textwidth]{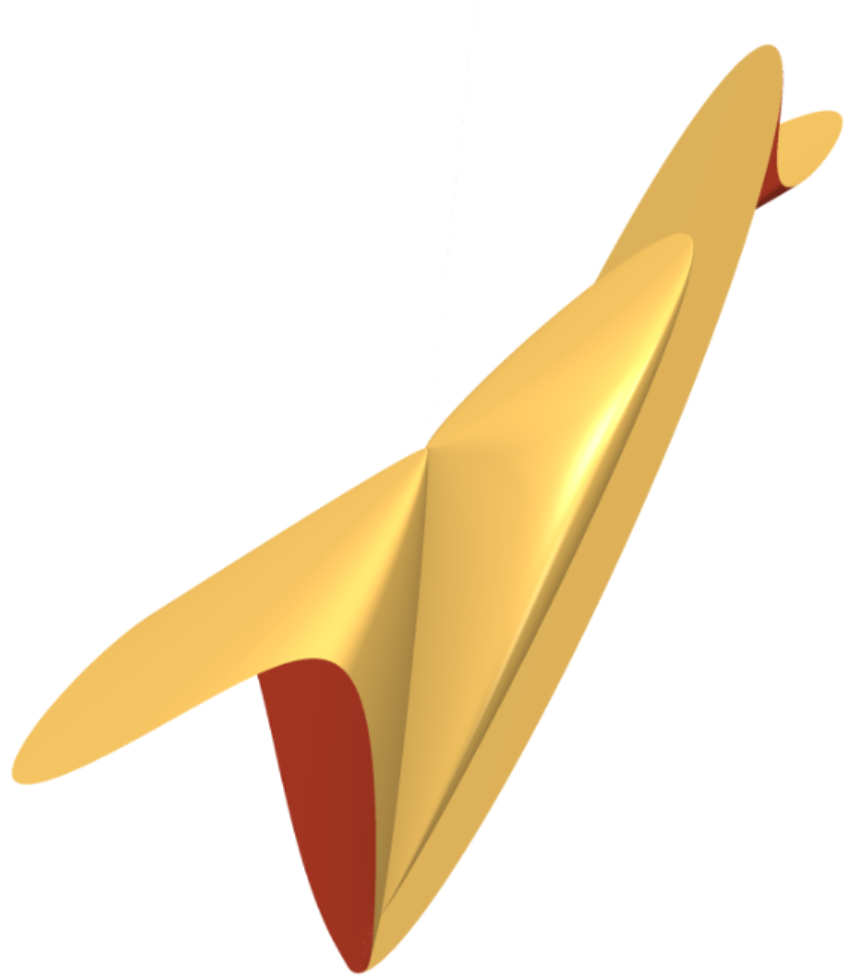}  
  \caption{Shifted Xano defined by $f = (x^4+(z+1)^3-(y+2)(z+1)^2)(y-1)$, and its dual defined by $f^* = (2 v_1+v_2-1)(v_0^4+128 v_1^4+320 v_1^3 v_2+256 v_1^2 v_2^2+64 v_1 v_2^3-64 v_1^3-128 v_1^2 v_2-64 v_1 v_2^2)$. } 
  \label{fig:dxano}  
\end{figure} 
The real variety of a shifted Xano 
{\ifcolor\color{magenta}\fi $X=\VV(f)$} 
with a plane defined by $\VV({y-1})$ is as Figure~\ref{fig:dxano}. 
{\ifcolor \color{purple}\fi 
The algebraic boundary $\partial_aK$ is 
$\VV(f)$. The singular locus of $\partial_aK$ is the union of $\VV(I_1),\VV(I_2)$ where
\begin{align*}
  I_1=&\langle {z+1,x}\rangle,
\\ 
 I_2=&\langle {y-1,x^4+z^3-3z-2}\rangle.
\end{align*}
The 2nd singular locus of $\VV(f)$ is $\VV(I_3)$ where
\begin{align*}
  I_3=&\langle {x,y-1,z+1}\rangle.
\end{align*} 
}
{\ifcolor\color{teal} \fi
The extreme point $p=(0,-2,-1)$ of $K$ lies on the line defined by $I_1$, but $p$ is not contained in any irreducible component the iterated singular loci. The dual of $p$ is the hyperplane $\VV(2v_1+v_2+1)$, which is an irreducible component of $\partial_a(K)^\dual$ {\upshape\cite[Corollary 3.9]{Sinn2015}}. } 
{\ifcolor\color{magenta}\fi Let $\overline{X}$ be the projective closure of $X$. We have $\overline{X}=\VV({g})$ with $g=(x^4+(z+w)^3w-(y+2w)(z+w)^2w)(y-w)\in \QQ[x,y,z,w] $. }
We follow the steps of Algorithm~\ref{alg:whitneya} for computing the Whitney~(a) stratification of the shifted Xano defined by $g$.
  {\ifcolor\color{purple}\fi
\begin{enumerate}
    \item[\quad 1-2.] Using the prime decomposition, we obtain 
        \begin{align*}
        \quad\quad    I=&\ideal{g}=\{\ideal{y-w}\cap \ideal{x^4+(z+w)^3 w-(y+2 w)(z+w)^2 w}, \\
        \quad\quad I_\text{sing}&=\ideal{z+w,x}\cap\ideal{y-w,x^4+z^3 w-3 z w^3-2 w^4}\cap\ideal{w,y-z,x}.
        \end{align*}
        Then we have $L=\{L_0,L_1\}$, with 
        \begin{align*}
         \quad\quad   L_0&=\{\ideal{y-w}, \ideal{x^4+(z+w)^3 w-(y+2 w)(z+w)^2 w} \}, \\
        \quad\quad L_1&=\{\ideal{z+w,x}, \ideal{y-w,x^4+z^3 w-3 z w^3-2 w^4}, \ideal{w,y-z,x}\}.
        \end{align*}
    \item[3.] By iterating and calling Algorithm~\ref{alg:notwhit}, we have
         \begin{align*}
         \quad\quad   L_2&=\{\ideal{z+w, y-w, x},  \ideal{w, z, x}, \ideal{z+w, y+2w, x}\}\}.
        \end{align*}
    \item[4.] Return 
        \begin{align*}
        L=&\{\{ \ideal{y-w}, \ideal{x^4+(z+w)^3 w-(y+2 w)(z+w)^2 w} \}, \\
        &\{\ideal{z+w,x}, \ideal{y-w,x^4+z^3 w-3z w^3-2 w^4},\ideal{w, y-z, x}\}, \\
        &\{\ideal{z+w, y-w, x},  \ideal{w, z, x}, \ideal{z+w, y+2w, x}\}\}
        \end{align*} 
\end{enumerate}
}
 {\ifcolor\color{magenta}\fi 
 Finally, taking $w=1$, we get 
 a Whitney (a) stratification of $\VV(f)$:
 \begin{itemize}
    \item[(1).] $F_0=\VV(f)$,
    \item[(2).] $F_1=\VV(x,z+1)\cup\VV(y-1,x^4+z^3-3 z-2)$,
    \item[(3).] $F_2=\{(0,1,-1),(0,-2,-1)\}$.
\end{itemize}  
 }
{\ifcolor\color{teal}\fi
The extreme point $p=(0,-2,-1)$ is in $F_2$.}
\end{example}

\subsection{Teardrop}

\begin{example}
\label{ex:teardrop}
Let $f=(z^2+y^2-(x+1)(x-1)^2)(y-5(x-1))(y+5(x-1))+\tfrac{1}{10}(x-1)yz^2\in\QQ[x,y,z]$, 
$I=\ideal{f}$, $X=\VV(I)$ and $\overline{X}=\VV(g)$, with $g=(z^2w+y^2w-(x+w)(x-w)^2)(y-5(x-w))(y+5(x-w))+\tfrac{1}{10}w(x-w)yz^2\in \QQ[x,y,z,w]$.
{\ifcolor\color{purple}\fi 
{\ifcolor\color{magenta}\fi
The following convex set $K'$ is a perturbation of the teardrop:
\begin{equation*}
    K'=\{(x,y,z)\in \RR^3\mid f\le 0, 5(x-1)\le y \le -5(x-1)\}.
\end{equation*}
The algebraic boundary of $K'$ 
}
is 
$\VV(f)$. The singular locus of $\partial_aK'$ is the union of $\VV(I_1),\VV(I_2),\VV(I_3)$ where
\begin{align*}
  I_1=&\langle {y,x-1}\rangle,
\\ 
 I_2=&\langle {z,-115+y,x-24}\rangle,
\\ 
 I_3=&\langle {z,115+y,x-24}\rangle.
\end{align*}
The 2nd singular locus of $\VV(f)$ is an empty set. }
{\ifcolor\color{teal}\fi
The extreme point $p=(1,0,0)$ lies on the line defined by $I_1$, but $p$ is not contained in any irreducible component of the iterated singular loci. 
The dual of $p$ is the hyperplane $\VV(x+1)$, which is an irreducible component of $\partial_a(K)^\dual$ {\upshape\cite[Corollary 3.9]{Sinn2015}}. }
For this example, computing the ideal of $C(X)$ suffers from coefficient swell, we use Macaulay2 and Maple to compute a Whitney~(a) stratification of $X$.   
  {\ifcolor\color{magenta}\fi Here are the compuation results: }
   {\ifcolor\color{purple}\fi 
\begin{enumerate}
    \item[1-2.] Using the prime decomposition, we obtain 
        \begin{align*}
            &I=\ideal{g}, \\
        &I_\text{sing}=\ideal{y, x-w}\cap\ideal{w, x, y^2+z^2} \\&\quad\qquad
        \cap\ideal{z, y-115w, x-24w}\cap\ideal{z, y+115w, x-24w}
        \end{align*}
        Then we have $L=\{L_0,L_1\}$, with 
        \begin{align*}
            L_0 &=\{\ideal{g} \}, \\            
         L_1 &= \{\ideal{y,x-w},\ideal{w,x,y^2+z^2}, 
         \\&\quad   \ideal{z, y-115w, x-24w},  \ideal{z, y+115w, x-24w}\}.
        \end{align*}
    \item[3.] By iterating and calling Algorithm~\ref{alg:notwhit}, we have
         \begin{align*}
         \quad\quad   L_2&=\{\ideal{w-x,y,z}, 
        \ideal{w,y,x}\}\}.
        \end{align*}  
    \item[4.] Return 
        \begin{align*}
        L=&\{\{ \ideal{g} \}, \\
       & \{\ideal{y,x-w},\ideal{w,x,y^2+z^2}, \ideal{z, y-115w, x-24w}, \\& \ideal{z, y+115w, x-24w}\}, \\
        &\{ \ideal{w-x,y,z}, 
        \ideal{w,y,x}\}\}
        \end{align*} 
\end{enumerate}}
{\ifcolor\color{magenta}\fi
Finally, taking $w=1$, we get a Whitney (a) stratification of $\VV(f)$:
\begin{itemize}
\item[(1).] $F_0=\VV(f)$,
\item[(2).] $F_1=\VV(x-1,y)\cup\{(24,115,0)\}\cup\{(24,-115,0)\}$,
\item[(3).] $F_2=\{(1,0,0)\}$.
\end{itemize}
}
{\ifcolor\color{teal}\fi
The extreme point $p=(1,0,0)$ is in $F_2$.}
\end{example}

\section*{Acknowledgements}
 Zijia Li and Lihong  Zhi are supported by the National Key R$\&$D Program of China (2023YFA1009401) and the National Natural Science Foundation of China (12071467).  Zhi-Hong Yang is supported by the National Natural
Science Foundation of China (12201425). Zijia Li is partially supported by the Strategic Priority Research Program of the Chinese Academy of Sciences 0640000 \& XDB0640200. The authors acknowledge the support of the Institut Henri Poincar\'e (UAR 839 CNRS-Sorbonne Universit\'e) and LabEx CARMIN (ANR-10-LABX-59-01).   We would like to express our sincere gratitude to Prof. Rainer Sinn for clarifying the question arising from the example of Teardrop. {\ifcolor\color{teal}\fi
 We thank the reviewers for their valuable comments. }

\bibliographystyle{abbrv}




\hspace{0.5cm}

%
%
%
%
%
%
%
%

KLMM, Academy of Mathematics and Systems Science, Chinese
Academy of Sciences, Beijing 100190, China, University of Chinese Academy of Sciences, Beijing 100049, China
\quad \emph{E-mail address}: \url{daizihao@amss.ac.cn};  \url{lizijia@amss.ac.cn}; \url{lzhi@mmrc.iss.ac.cn}

School of Mathematics and Statistics, Central South  University,Changsha,China
 \quad
\emph{E-mail address}: \url{yangzhihong@csu.edu.cn}


%

\end{document}